\documentclass[11pt,reqno]{amsart}
\usepackage{amsmath,amsthm,amssymb,amscd,url,enumerate,mathtools}
\usepackage{graphicx}
\usepackage[margin=1.5in]{geometry}
\usepackage{afterpage}
\usepackage{tikz}
\usepackage{todonotes}
\usepackage[all]{xy}
\usepackage[colorlinks=true]{hyperref}
\usepackage{tabularx}

\usepackage[square]{natbib}
\setcitestyle{numbers}

\newcommand{\degr}{\operatorname{deg}}
\newcommand{\disc}{\operatorname{disc}}
\newcommand{\ind}{\operatorname{ind}}

\newcommand{\dnd}{\nmid}

\newcommand{\TITLE}{Two Families of Monogenic $S_4$ Quartic Number Fields}
\newcommand{\TITLERUNNING}{}

\theoremstyle{plain}
\newtheorem{theorem}{Theorem}

\newtheorem{corollary}[theorem]{Corollary}

\theoremstyle{definition}
\newtheorem{definition}[theorem]{Definition}

\theoremstyle{remark}
\newtheorem{remark}[theorem]{Remark}

\newtheorem{example}[theorem]{Example}

\numberwithin{theorem}{section}

  {\end{list}}

  {\end{list}}

\newcommand{\tightoverset}[2]{%
  \mathop{#2}\limits^{\vbox to -.5ex{\kern-1.05ex\hbox{$#1$}\vss}}}

\def\Ocal{{\mathcal O}}

\newcommand{\FF}{\mathbb{F}}

\newcommand{\QQ}{\mathbb{Q}}

\newcommand{\ZZ}{\mathbb{Z}}

\newcommand{\inv}{^{-1}}

\title[\TITLERUNNING]{\vspace*{-1.3cm} \TITLE}

\date{\today}

\author[Hanson Smith]{Hanson Smith}
\address{%
Department of Mathematics, University of Colorado,
Campus Box 395, Boulder, Colorado 80309-0395}
\email{hanson.smith@colorado.edu}


\thanks{}

\begin{document}

\begin{abstract}
  Consider the integral polynomials $f_{a,b}(x)=x^4+ax+b$ and $g_{c,d}(x)=x^4+cx^3+d$. Suppose $f_{a,b}(x)$ and $g_{c,d}(x)$ are irreducible, $b\mid a$, and the integers $b$, $d$, $256d-27c^4$, and $\dfrac{256b^3-27a^4}{\gcd(256b^3,27a^4)}$ are all square-free. Using the Montes algorithm, we show that a root of $f_{a,b}(x)$ or $g_{c,d}(x)$ defines a monogenic extension of $\QQ$ and serves as a generator for a power basis of the ring of integers. In fact, we show monogeneity for slightly more general families. Further, we obtain lower bounds on the density of polynomials generating monogenic $S_4$ fields within the families $f_{b,b}(x)$ and $g_{1,d}(x)$. 
\end{abstract}

\maketitle

\section{Introduction and Overview of Results}

Let $K$ be a number field and let $\Ocal_K$ be its ring of integers. If there exists a monic irreducible polynomial $f(x)\in \ZZ[x]$ with a root $\theta$ such that $\ZZ[\theta]=\Ocal_K$, then we say $K$ is \emph{monogenic}. In other words, $K$ is monogenic if $\Ocal_K$ admits a power integral basis. A quantity related to monogeneity is the field index. The \emph{field index} is defined to be the pair-wise greatest common divisor $\gcd\limits_{\alpha\in \Ocal_K} \left[\Ocal_K:\ZZ[\alpha]\right]$. Note that $K$ can have field index 1 and still not be monogenic. Define the \emph{minimal index} to be $\min \limits_{\alpha\in \Ocal_K} \left[\Ocal_K:\ZZ[\alpha]\right]$. Monogeneity is equivalent to having minimal index equal to 1.
 
Many of the number fields we are most familiar with are monogenic. For example, all quadratic extensions and cyclotomic extensions are monogenic. An example of a non-monogenic field, due to Dedekind \cite{Dedekind}, is the field obtained by adjoining a root of $x^3-x^2-2x-8$ to $\QQ$. The problem of classifying monogenic number fields is often called \emph{Hasse's problem}, as it is believed to have been posed to the London Mathematical Society by Helmut Hasse in the 1960's. See the remark on page 193 of \cite{Hasse}. 

We can now state concise, less-general versions of our main results. The following are consequences of Theorems \ref{fmon} and \ref{f1} and Theorems \ref{gmon} and \ref{g1}, respectively.

\begin{corollary}\label{corf} Consider $f_{b,b}(x)=x^4+bx+b$ with $b\in \ZZ$ and let $\theta$ be a root. Suppose $b$ and $256-27b$ are square-free and $b\neq 3,5$. Then, $\QQ(\theta)$ is a monogenic $S_4$ quartic field and $\theta$ is a generator of a power integral basis. Further, at least 29.18\% of $b\in \ZZ$ satisfy these conditions.
\end{corollary}

\begin{corollary}\label{corg} Consider $g_{1,d}(x)=x^4+x^3+d$ with $d\in \ZZ$ and let $\tau$ be a root. Suppose $d$ and $256d-27$ are square-free and $d\neq -2$. Then $\QQ(\tau)$ is a monogenic $S_4$ quartic field and $\tau$ is a generator of a power integral basis. Further, at least 41.849\% of $d\in \ZZ$ satisfy these conditions.
\end{corollary}

Heuristically, the best possible percentages in Corollaries \ref{corf} and \ref{corg} seem to be 55.3\%. See Remark \ref{data}. 

\section{Previous Work}

Before a more detailed exposition of our work, we list some results pertaining to Hasse's problem. It has been shown that almost all abelian extension of $\QQ$ with degree coprime to 6 are not monogenic; see Gras \cite{Gras}. Gassert \cite{Alden} shows that all fields obtained by adjoining a root of $x^n-a$, where $a$ is square-free and $a^p\not \equiv a$ modulo $p^2$ for all primes $p\mid n$, are monogenic. In \cite{JonesPhillips}, Jones and Phillips identify infinitely many monogenic fields coming from polynomials of the shape $x^n+a(m,n)x+b(m,n)$, where $a(m,n)$ and $b(m,n)$ are prescribed forms. They consider two families of forms, one yielding Galois group $S_n$ and the other $A_n$. Recently, Bhargava, Shankar, and Wang \cite{BSW} have shown that the density of monic, irreducible polynomials $f(x)\in\ZZ[x]$ such that a root, $\theta$, of $f(x)$ yields a power basis for the ring of integers of $\QQ(\theta)$ is $\frac{6}{\pi^2}=\zeta(2)\inv\approx 60.79\%$. In the same paper, they also show that the density of monic integer polynomials with square-free discriminants is 
$$\prod\limits_p \left(1-\dfrac{1}{p}+\dfrac{(p-1)^2}{p^2(p+1)}\right)\approx 35.82\%.$$
Note that these polynomials are a subset of the monic, irreducible polynomials $f(x)\in \ZZ[x]$ such that a root, $\theta$, yields a power basis for the ring of integers of $\QQ(\theta)$. 

Many of the approaches to Hasse's problem have focused on fields with a given Galois group. We summarize the state of the art for degree 4 number fields. For a nice treatise on approaches to monogeneity using index form equations, see Ga\'al's book \cite{GaalDiophantine}, ``Diophantine Equations and Power Integral Bases." In particular, Chapter 6 deals with the quartic case. A general algorithm for solving the quartic index form equations is presented, with the author expanding upon specific cases. 

A biquadratic field is an extension having $\ZZ/2\ZZ\times \ZZ/2\ZZ$ as the Galois group. A biquadratic extension can be written $\QQ\left(\sqrt{m},\sqrt{n}\right)$. Ga\'al, Peth\H{o}, and Pohst \cite{GaalPethoPohst} parametrize the field indices that can occur based on congruence conditions on $m$ and $n$. Gras and Tano\'e's article \cite{tangras} gives necessary and sufficient conditions for a biquadratic field to be monogenic. Jadrijevi\'{c} \cite{Jadrijevic} describes the minimal and field indices of the two families $\QQ\left(\sqrt{(c-2)c},\sqrt{(c+2)c}\right)$ and $\QQ\left(\sqrt{(c-2)c},\sqrt{(c+4)c}\right)$. This investigation is continued for the family $\QQ\left(\sqrt{(c-2)c},\sqrt{(c+4)c}\right)$ in \cite{Jadrijevic1}. When $c$ and $c+4$ are square-free, Ga\'al and Jadrijevi\'{c} \cite{GaalJadrijevic} show $\QQ\left(\sqrt{2c},\sqrt{2(c+4)}\right)$ is not monogenic, compute an integral basis, and determine the elements of minimal index. 

Dihedral quartic fields have received a significant amount of attention. In \cite{HSW}, Huard, Spearman, and Williams compute the discriminant and an integral basis of quartic fields with quadratic subfields. Further, they find infinitely many monogenic $D_8$ fields. Specifically, they show that for each square-free $c$ there are infinitely many fields of the form $\QQ\left(\sqrt{c},\sqrt{a+b\sqrt{c}}\right)$ that are monogenic. Ga\'al and Szab\'o \cite{GaalSzabo} solve index form equations to show that the power integral bases found in \cite{HSW} are the only possible power integral bases. In \cite{Kable}, Kable resolves the question of monogeneity when the $D_8$ field in question has an imaginary quadratic subfield and establishes some bounds in all cases. Using their algorithm from \cite{GPPI}, Ga\'al, Peth\H{o}, and Pohst \cite{GPPindexform} compute ``small" indices of totally real quartic fields with Galois group either $\ZZ/4\ZZ$ or $D_8$ and discriminant of absolute value less than $10^6$. The indices may not be minimal since the algorithm they implemented checks only for solutions to the index form equation with absolute value less than $10^6$.

A \emph{pure} quartic field is a field obtained by adjoining a root of a polynomial of the form $x^4-a$ to $\QQ$. In \cite{Funa}, Funakura gives necessary and sufficient conditions for pure monogenic quartic fields. Ga\'{a}l and Remete \cite{GaalRemete}, characterize the only power integral bases of a number of infinite families of pure quartic fields using binomial Thue equations and extensive calculations on a supercomputer. 

The \emph{simplest} quartic fields are given by a root of $x^4-tx^3-6x^2+tx+1$, where $t\neq \pm3,0$. They are totally real with Galois group $\ZZ/4\ZZ$. If $t^2+16$ is not divisible by an odd square, Olajos \cite{Olajos} has shown that the only two simplest quartics that are monogenic occur when $t=2$ and $t=4$. In \cite{GrasZbases}, Gras shows there are only two monogenic imaginary cyclic quartic fields. These are $\QQ(\zeta_5)$ and $\QQ\left(\zeta_{16}-\zeta_{16}\inv\right)$.

For $A_4$ fields, Spearman \cite{spear} shows $x^4+18x^2-4tx+t^2+81$ defines an infinite family of monogenic fields when $t\left(t^2+108\right)$ is square-free. 

With \cite{Gaalquartic}, Ga\'al considers five families of totally complex quartic polynomials. The polynomials are shown to be irreducible and the Galois groups are classified; $A_4$, $D_8$, $\ZZ/4\ZZ$, and $\ZZ/2\ZZ\times \ZZ/2\ZZ$ all occur. Further, Ga\'al computes all power integral bases of the orders generated by the roots. 

As for $S_4$ quartics, work by B\'erczes, Evertse, and Gy\H{o}ry \cite{BEG} restricts multiply monogenic orders. A recent paper \cite{GSS} by Gassert, Smith, and Stange shows $x^4-6x^2-tx-3$ with $t+8$ and $t-8$ square-free defines an infinite family of monogenic $S_4$ quartic fields. It is worth noting that the methods of \cite{GSS} are distinct from much of the other literature in that arithmetic properties of elliptic curves are central to proving monogeneity.

\section{Results}
In this paper we identify two families of monogenic quartic fields:

\begin{theorem}\label{fmon}
Let $a$ and $b$ be integers such that $\dfrac{256b^3-27a^4}{\gcd(256b^3,27a^4)}$ is square-free. Suppose that $f_{a,b}(x)=x^4+ax+b$ is irreducible and let $\theta$ be a root. Further, suppose every prime, $p$, dividing $\gcd(256b^3,27a^4)$ satisfies one of the following conditions:
\begin{enumerate}

\item $p$ divides $a$ and $b$, but $p^2$ does not divide $b$.

\item $p=2$, $p\dnd b$, and $(a,b)$ is congruent to one of the following pairs in $\ZZ/4\ZZ\times \ZZ/4\ZZ$:  $(0,1)$, $(2,3)$.

\item $p=3$, $p\dnd a$, and $(a,b)$ is congruent to one of the following pairs in $\ZZ/9\ZZ\times \ZZ/9\ZZ$:  $(1,3)$, $(1,6)$, $(2,0)$, $(2,3)$, $(4,0)$, $(4,6)$, $(5,0)$, $(5,6)$, $(7,0)$, $(7,3)$, $(8,3)$, $(8,6)$. 
 
\end{enumerate}
Then, $\QQ(\theta)$ is monogenic and $\theta$ is a generator of the ring of integers.
\end{theorem}
 
\begin{theorem}\label{gmon}
Let $c$ and $d$ be integers such that $d$ is square-free and $256d-27c^4$ is not divisible by the square of an odd prime. If $4\mid \left(256d-27c^4\right)$, we require that $(c,d)$ is congruent to either $(0,1)$ or $(2,3)$ in $\ZZ/4\ZZ\times \ZZ/4\ZZ$. Suppose that $g_{c,d}(x)=x^4+cx^3+d$ is irreducible and let $\tau$ be a root. Then, $\QQ(\tau)$ is monogenic and $\tau$ is a generator of the ring of integers.
\end{theorem} 

If we restrict the above families we can classify the Galois groups and analyze densities. Note the infinitude of the restricted families below shows the more general families described above are infinite.
 
\begin{theorem}\label{f1}
With the notation as in Theorem \ref{fmon}, consider $f_{b,b}(x)=x^4+bx+b$. Suppose the coefficients of $f_{b,b}(x)$ satisfy the conditions given in Theorem \ref{fmon}. If $b\neq 3,5$, then $\QQ(\theta)$ has Galois group $S_4$. Moreover, the density of polynomials satisfying Theorem \ref{fmon} among polynomials of the form $x^4+bx+b$ with $b\in \ZZ$ arbitrary is at least $\dfrac{51-4\pi^2}{4\pi^2}\approx 29.18\%$.
\end{theorem}

\begin{theorem}\label{g1}
With the notation as in Theorem \ref{gmon}, consider $g_{1,d}(x)=x^4+x^3+d$. Suppose the coefficients of $g_{1,d}(x)$ satisfy the conditions given in Theorem \ref{gmon}. If $d\neq -2$, then $\QQ(\tau)$ has Galois group $S_4$. Moreover, the density of polynomials satisfying Theorem \ref{gmon} among polynomials of the form $x^4+x^3+d$ with $d\in \ZZ$ arbitrary is at least $\dfrac{14-\pi^2}{\pi^2}\approx 41.85\%$.
\end{theorem}

The primary reason for choosing the restricted families in the above theorems was so that we could easily analyze their densities. Within the larger class of polynomials which we prove yield monogenic fields, one can find other restrictions on the coefficients that yield families with a specific Galois group. However, in these cases studying density becomes more difficult, as one is concerned with square-free values of higher degree polynomials. Our methods could achieve similar results for polynomials of the form $x^4+ax^2+b$ or $x^4+c$. However, these families have already been well-studied.

The outline of our paper is as follows:  To prove Theorems \ref{fmon} and \ref{gmon} our main tool is the Montes algorithm, which we will briefly describe in Section \ref{Montessection}. In Section \ref{Galoisandirred}, we show that the restricted families are irreducible and have Galois group $S_4$. Section \ref{monogeneity} is concerned with applying the Montes algorithm to prove monogeneity. Lastly, in Section \ref{Density}, we analyze the densities of our restricted families. 

\subsection*{Acknowledgements}  The author would like to thank Katherine Stange and Alden Gassert for their help and encouragement. The author would also like to thank Sebastian Bozlee for the careful proofreading. 

\section{The Montes Algorithm}\label{Montessection}
We prove monogeneity with a simple application of the Montes algorithm. We follow \cite{EFMN} for our exposition of the algorithm. Those interested in more general situations are advised to consult \cite{gmn15}. For the purposes of our work, the goal of the Montes algorithm is to compute the $p$-adic valuation $v_p([\Ocal_K:\ZZ[\theta]])$. 

We begin by fixing notation. Let $f(x)\in\ZZ[x]$ be monic and irreducible, $\theta$ a root of $f(x)$, $K=\QQ(\theta)$, $\Ocal_K$ the ring of integers of $K$, and $p$ a prime in $\ZZ$. We extend the $p$-adic valuation on $\ZZ$ to $\ZZ[x]$ in the following manner. If $g(x)=b_0+b_1 x+\cdots +b_k x^k$, define $v_p(g(x))=\min\limits_{0\leq j\leq k}(v_p(b_j))$.

Now we describe a version of the Montes algorithm. Consider the reduction of $f(x)$ modulo $p$. Let $\overline{\phi}(x)$ be an irreducible factor of $f(x)$ modulo $p$ and let $\phi(x)$ be a lift of $\overline{\phi}(x)$ to $\ZZ[x]$. We may write
$$f(x)=a_0(x)+a_1(x)\phi(x)+\cdots +a_r(x)\phi(x)^r$$
where $a_i(x)\in\ZZ[x]$ has degree strictly less than $\degr(\phi(x))$. We call this the \emph{$\phi$-adic development} of $f$. To any coefficient, $a_i(x)$, of the $\phi$-adic development of $f$ we attach the point $(i, v_p(a_i(x)))$ in the plane. The lower convex envelope of these points is called the \emph{$\phi$-Newton polygon of $f$}. The polygon determined by the sides of the $\phi$-Newton polygon with negative slope is called the \emph{principal $\phi$-polygon of $f$}. We denote this polygon by $N$. The integer lattice points on or below $N$ contain the arithmetic information we are interested in. Specifically, the \emph{$\phi$-index of $f$} is $\degr(\phi)$ times the number of points in the plane with integral coordinates that lie on or below $N$, strictly above the $x$-axis, and strictly to the right of the $y$-axis. We denote this number, the number of points in the integer lattice satisfying the above conditions, by $\ind_\phi(f)$. 

\begin{example}
To illustrate how the $\phi$-Newton polygon is obtained, consider $f(x)=x^6+3x^5+x^4+15x^3+9x^2+18x+27$. We reduce modulo $3$ and obtain $x^4(x^2+1)$. Working with the irreducible factor $x$, we take the lift $x$ and the $x$-adic development is again our original polynomial $f(x)=x^6+3x^5+x^4+15x^3+9x^2+18x+27$. Now the $x$-Newton polygon is:
\begin{center}
\begin{tikzpicture}[scale=1]
    \draw[black, very thick] (4,0) -- (6,0);  
    \draw [<->,thick] (0,4)
        |- (7,0);
    \draw[black, thick] (0,3) -- (1,2);
    \draw[black, thick] (1,2) -- (4,0);
    \draw[black, thick] (4,0) -- (6,0);    
    \filldraw[black] (0,3) circle (2pt) node[anchor=east] {(0,3)};
    \filldraw[black] (1,2) circle (2pt) node[anchor=east] {(1,2)};
    \filldraw[black] (2,2) circle (2pt) node[anchor=west] {(2,2)};
    \filldraw[black] (3,1) circle (2pt) node[anchor=south] {(3,1)};
    \filldraw[black] (4,0) circle (2pt) node[anchor=north] {(4,0)};
    \filldraw[black] (5,1) circle (2pt) node[anchor=south] {(5,1)};
    \filldraw[black] (6,0) circle (2pt) node[anchor=north] {(6,0)};
    \filldraw[red] (2,1) circle (2pt) node[anchor=north] {(2,1)};
    \filldraw[red] (1,1) circle (2pt) node[anchor=north] {(1,1)};
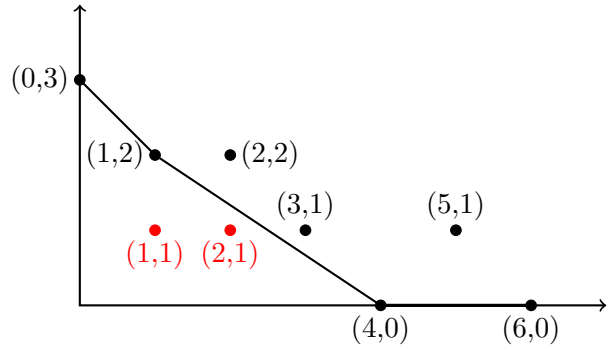
    \node[] at (3.5,-1) {The $x$-Newton polygon for $f(x)$};
\end{tikzpicture}
\end{center}
The principal $x$-polygon merely excludes the side between $(4,0)$ and $(6,0)$. Further, accounting for $(1,1)$, $(2,1)$, and $(1,2)$, we see $\ind_x(f)=3$.
\end{example}

Continuing with our description of the Montes algorithm, to any integral $x$-coordinate $0\leq i\leq r$ of the principal $\phi$-polygon $N$, we attach the residual coefficient $c_i\in \FF_p[x]/\phi(x)$, defined to be
$$c_i=\left\{\begin{array}{lr}
        0, & \text{if } (i,v_p(a_i(x))) \text{ lies strictly above $N$}\\
        & \text{or $v_p(a_i(x))=\infty$.}\\
        &\\
        \dfrac{a_i(x)}{p^{v_p(a_i(x))}}\in \FF_p[x]/\phi(x), & \text{if $(i,v_p(a_i(x)))$ lies on $N$.}
        \end{array}\right\}$$
Note we have covered all cases since $(i,v_p(a_i(x)))$ cannot lie below $N$, as $N$ is the lower convex hull of the $(i,v_p(a_i(x)))$.

Let $S$ be one of the sides of $N$. Suppose $S$ has slope $\lambda=\dfrac{-h}{e}$ where $h,e$ are positive, coprime integers. Define the \emph{length} of $S$, denoted $l$, to be the length of the projection onto the $x$-axis. The \emph{ramification index} of $S$ is $e$, the denominator of $\lambda$. The \emph{degree} of $S$, denoted $d$, is $\dfrac{l}{e}$.
\begin{definition}
Let $t$ be the $x$-coordinate of the initial vertex of $S$. We define the \emph{residual polynomial} attached to $S$ to be 
$$R_\lambda(f)(y)=c_t+c_{t+e}y+\cdots+c_{t+(d-1)e}y^{d-1}+c_{t+de}y^d\in \FF_p[x]/\phi(x)[y].$$
\end{definition} 

Now we state the Theorem of the index, our key tool in proving monogeneity. This is Theorem 1.9 of \cite{EFMN}.
\begin{theorem}\label{Thmofindex}
Choose monic polynomials $\phi_1,\dots, \phi_k$ whose reduction modulo $p$ are the different irreducible factors of $f(x)$. Then, $$v_p([\Ocal_K:\ZZ[\theta]])\geq \ind_{\phi_1}(f)+\cdots + \ind_{\phi_k}(f).$$
Further, equality holds if and only if, for every $\phi_i$, each side of the principal $\phi_i$-polygon has a separable residual polynomial.
\end{theorem}

\begin{remark}\label{sep} The Montes algorithm is concerned with separability. With the notation as above, suppose $f(x)\equiv \gamma(x)\psi(x)$ modulo $p$ where $\gamma(x)$ is separable and $\gcd(\gamma(x),\psi(x))=1$. Then, $\gamma(x)$ contributes nothing to $v_p([\Ocal_K:\ZZ[\theta]])$. To see this, let $\eta(x)$ be an irreducible factor of $\gamma(x)$ and consider the $\eta(x)$-adic development of $f(x)$:
$$f(x)=a_0(x)+a_1(x)\eta(x)+\cdots +a_r(x)\eta(x)^r.$$
Because $f(x)$ has only one factor of $\eta(x)$ modulo $p$, we note $p\dnd a_1(x)$. Hence the principal $\eta$-polygon has only one side and that side terminates at $(1,0)$. Thus $\ind_\eta(f)=0$. Furthermore, the residual polynomial will be separable since linear polynomials are always separable. 

\end{remark}

\section{Galois Groups and Irreducibility}\label{Galoisandirred}
Consider the two families $f_{a,b}(x)=x^4+ax+b$ and $g_{c,d}(x)=x^4+cx^3+d$. These polynomials have discriminants $\Delta_f=b^3(256-27a^4b)$ and $\Delta_g=d^2(256d-27c^4)$. To prove monogeneity, we require the conditions outlined in Theorems \ref{fmon} and \ref{gmon}. However, to obtain families with Galois group $S_4$, we impose further restrictions. Namely, we require $a=b\neq 3,5$ for $f_{a,b}(x)$ and $c=1$, $d\neq -2$ for $g_{c,d}(x)$. There are less restrictive $S_4$ families, but we have chosen these parameters so that we can analyze the densities of these families.

In this section we are concerned with proving the first claims of Theorems \ref{f1} and \ref{g1}, which we restate.

\begin{theorem}\label{Gal}
The polynomials $f_{b,b}(x)=x^4+bx+b$ and $g_{1,d}(x)=x^4+x^3+d$ where $b$, $d$, $256-27b$, and $256d-27$  are square-free, $b\neq 3,5$, and $d\neq -2$ are irreducible and have Galois group $S_4$.
\end{theorem}

Before proving Theorem \ref{Gal}, we state two results we will need. We begin with some definitions. Given a quartic polynomial $h(x)=x^4+a_3x^3+a_2x^2+a_1x+a_0$ with roots $\alpha_1,\alpha_2,\alpha_3,\alpha_4$, we define the \emph{resolvent cubic} to be
$$R_h(y)=y^3-a_2y^2+(a_3a_1-4a_0)y-a_3^2a_0-a_1^2+4a_2a_0.$$

$R_h$ has roots $\alpha_1\alpha_2+\alpha_3\alpha_4$, $\alpha_1\alpha_3+\alpha_2\alpha_4$, and $\alpha_1\alpha_4+\alpha_2\alpha_3$. Given $h(x)$, a \emph{depressed quartic} is obtained by the substitution $x=X-\dfrac{a_3}{4}$ and has the form
$$h_{dep}(X)=X^4+\left(\dfrac{-3a_3^2}{8}+a_2\right)X^2+\left(\dfrac{a_3^3}{8}-\dfrac{a_3a_2}{2}+a_1\right)X+\left(-\dfrac{3a_3^4}{256}+\dfrac{a_3^2a_2}{16}-\dfrac{a_3a_1}{4}+a_0\right).$$
If we have a depressed quartic $h_{dep}(x)=x^4+b_2x^2+b_1x+b_0$, we define the \emph{resolvent cubic} to be 
$$R_{h,dep}(z)=z^3+2b_2z^2+\left(b_2^2-4b_0\right)z-b_1^2.$$
Though $R_h(y)$ and $R_{h,dep}(z)$ are both called the resolvent cubic, they are actually different polynomials even if $h(x)$ is depressed to begin with. More specifically, the substitution $y=z-\dfrac{a_3^2}{4}+a_2$ sends $R_h(y)$ to $R_{h,dep}(z)$. Thus $R_h$ has a root in $\QQ$ if and only if $R_{h,dep}$ has a root in $\QQ$.

Now we recall a classical theorem. One can see \cite{LostArt} for a clear, elementary exposition.
\begin{theorem}\label{LostArt}
With the notation as above, $h(x)$ factors into quadratic polynomials in $\QQ[x]$ if and only if at least one of the following hold:
\begin{enumerate}
\item\label{gal1} $R_{h,dep}$ has a nonzero root in $\QQ^2$. That is, $R_{h,dep}$ has a root that is the square of a nonzero rational number.
\item\label{gal2} $b_1=0$ and $b_2^2-4b_0\in \QQ^2$.
\end{enumerate}
\end{theorem}

We will also use the following result of Kappe and Warren \cite[Theorem 1]{GalQuar} to determine the Galois groups. One can also consult \cite{KCBlurb} for a nice exposition with ample examples.

\begin{theorem}\label{GalQuar} Let $h(x)$ be a quartic polynomial that is irreducible over $\QQ$ and let $\Delta_h$ be the discriminant. Further, let $G_h$ be the Galois group of $h$. Then, with the notation as above, the first two columns of the following table imply the third column.
\begin{center}
\begin{tabular}{ |c|c|c| } 
 \hline
 $\Delta_h$ & $R_h$ & $G_h$\\
 \hline
 not a square & irreducible & $S_4$ \\ 
 a square & irreducible & $A_4$ \\ 
 not a square & reducible & $D_8$ or $\ZZ/4\ZZ$ \\
 a square & reducible & $\ZZ/2\ZZ \times \ZZ/2\ZZ$ \\ 
 \hline
\end{tabular}
\end{center} 
\end{theorem}

We proceed with the proof of Theorem \ref{Gal}.
\begin{proof} 
We begin with the irreducibility $f_{b,b}(x)=x^4+bx+b$. If $b$ is not $\pm 1$, then $f_{b,b}$ is Eisenstein at any prime dividing $b$. If $b=\pm 1$, then the rational root test shows that there is not a root in $\QQ$. To show $f_{\pm 1,\pm 1}$ does not split into quadratic factors we consider $R_{f_{\pm1,\pm1},dep}(z)=z^3\mp 4z-1$. The rational root test shows $R_{f_{\pm1,\pm1},dep}$ does not have a root in $\QQ$, let alone $\QQ^2$. Since $b_2=\pm1$, Theorem \ref{LostArt} shows $f_{\pm1,\pm1}$ is irreducible. Note that since $R_{f_{\pm1,\pm1},dep}$ is irreducible, $R_{f_{\pm1,\pm1}}$ is irreducible.

It remains to consider $R_{f_{b,b}}=y^3-4by-b^2$ for $b\neq \pm 1$. Suppose we have a root $k$. The rational root test shows $k\in \ZZ$ divides $b^2$. If $p$ is a prime divisor of $b$, then reduction modulo $p$ shows $p$ divides $k$. Since $b$ is square-free, $b$ divides $k$. Write $k =bk_0$. We have $b^2(bk_0^3-4k_0-1)=0$. Thus, $bk_0^3-4k_0-1=0$. If any prime divides $k_0$, reduction modulo that prime yields a contradiction. Hence, $k_0=\pm 1$. If $k_0=1$, then $b=5$, and if $k_0=-1$, then $b=3$. Therefore, if $b\neq 3,5$, then $R_{f_{b,b}}$ is irreducible.  

For $g_{1,d}(x)=x^4+x^3+d$, suppose we have a root $l$. By Gauss's lemma, $l\in \ZZ$. We have $l^3(l+1)=-d$. Since $d$ is square-free, we must have $l=1$ and $d=-2$. Thus for $d\neq -2$, we conclude $g_{1,d}$ does not have a root in $\QQ$. To see that $g_{1,d}$ does not factor into quadratics, we make the change of variables $x=X-\dfrac{1}{4}$ to obtain the depressed quartic 
$$X^4-\dfrac{3}{8}X^2+\dfrac{1}{8}X-\dfrac{3}{256}+d.$$
Here $b_1=\dfrac{1}{8}$ so condition \eqref{gal2} of Theorem \ref{LostArt} does not hold. Consider $R_{g_{1,d}}(y)=y^3-4dy-d$. If $d\neq \pm 1$, then $R_{g_{1,d}}$ is Eisenstein at any prime dividing $d$ and hence irreducible. If $d=\pm 1$, the rational root test shows $R_{g_{1,d}}$ is irreducible. Thus condition \eqref{gal1} of Theorem \ref{LostArt} does not hold since $R_{g_{1,d}}$ has a root in $\QQ$ if and only if $R_{g_{1,d},dep}$ has a root in $\QQ$. We conclude $g_{1,d}$ is irreducible.

We have demonstrated that, with the conditions given in Theorem \ref{Gal}, $f_{b,b},g_{1,d},R_{f_{b,b}},$ and $R_{g_{1,d}}$ are all irreducible. A quick computation shows that $\Delta_{f_{b,b}}$ and $\Delta_{g_{1,d}}$ are not squares. Thus, Theorem \ref{GalQuar} shows that $f_{b,b}$ and $g_{1,d}$ have Galois group $S_4$.  

\end{proof}

\begin{remark}\label{monocubic} Let $\beta$ be a root of $R_{g_{1,d}}$. Note that $g_{1,d}$ and $R_{g_{1,d}}$ both have discriminant $d^2(256d-27)$. The methods of Section \ref{monogeneity} show $\beta$ generates a power basis for the ring of integers of the cubic field $\QQ(\beta)$ exactly when $\tau$, a root of $g_{1,d}$, generates a power basis for the ring of integers of the quartic field $\QQ(\tau)$. 
\end{remark}

\section{monogeneity}\label{monogeneity}

For the following we recall a classical formula from algebraic number theory. Let $K$ be a number field obtained by adjoining a root, $\alpha$, of some monic irreducible polynomial $h(x)\in \ZZ[x]$. Write $\mathcal{O}_K$ for the ring of integers, $\disc(K)$ for the discriminant of $K$, and $\Delta_h$ for the discriminant of $h(x)$. Let $p$ be a prime. We have
\begin{equation*}\label{classic}
v_p(\disc(K))+2v_p([\mathcal{O}_K:\ZZ[\alpha]])=v_p(\Delta_h).
\end{equation*}
Note this implies any prime dividing $[\mathcal{O}_K:\ZZ[\alpha]]$ also divides $\Delta_h$. Before we proceed with the proof, we recall Theorem \ref{fmon}:
  
\begin{theorem}\label{A}
Let $a$ and $b$ be integers such that $\dfrac{256b^3-27a^4}{\gcd(256b^3,27a^4)}$ is square-free. Suppose that $f_{a,b}(x)=x^4+ax+b$ is irreducible and let $\theta$ be a root. Further, suppose every prime, $p$, dividing $\gcd(256b^3,27a^4)$ satisfies one of the following conditions:
\begin{enumerate}

\item\label{1} $p$ divides $a$ and $b$, but $p^2$ does not divide $b$.

\item\label{2} $p=2$, $p\dnd b$, and $(a,b)$ is congruent to one of the following pairs in $\ZZ/4\ZZ\times \ZZ/4\ZZ$:  $(0,1)$, $(2,3)$.

\item\label{3} $p=3$, $p\dnd a$, and $(a,b)$ is congruent to one of the following pairs in $\ZZ/9\ZZ\times \ZZ/9\ZZ$:  $(1,3)$, $(1,6)$, $(2,0)$, $(2,3)$, $(4,0)$, $(4,6)$, $(5,0)$, $(5,6)$, $(7,0)$, $(7,3)$, $(8,3)$, $(8,6)$. 
 
\end{enumerate}
Then, $\QQ(\theta)$ is monogenic and $\theta$ is a generator of the ring of integers.
\end{theorem}

\begin{proof}
Recall $\Delta_f=256b^3-27a^4$. Let $p$ be a prime dividing $\Delta_f$. We will show that $v_p([\mathcal{O}_{\QQ(\theta)}:\ZZ[\theta]])=0$. First, suppose $p\mid \Delta_f$, but $p\dnd \gcd(256b^3,27a^4)$. Since $\dfrac{256b^3-27a^4}{\gcd(256b^3,27a^4)}$ is square-free, we see 
$$1=v_p(\Delta_f)=v_p(\disc(\QQ(\theta)))+2v_p([\mathcal{O}_{\QQ(\theta)}:\ZZ[\theta]]).$$
Thus $v_p([\mathcal{O}_{\QQ(\theta)}:\ZZ[\theta]])=0$. 

So we consider primes $p$ dividing $\gcd(256b^3,27a^4)$. Suppose $p$ satisfies condition \eqref{1}. We apply the Montes algorithm. Considering $f_{a,b}(x)$ modulo $p$ we obtain $x^4$. Thus the only irreducible factor we must consider is $x$. Taking the lift $\phi(x)=x$, the principal $x$-polygon of $f_{a,b}(x)$ has one side, originating at $(0,1)$ and terminating at $(4,0)$. Thus $\ind_x(f_{a,b})=0$. The residual polynomial attached to this side is $y-\frac{b}{p}$, which is clearly separable. By Theorem \ref{Thmofindex}, $v_p([\mathcal{O}_{\QQ(\theta)}:\ZZ[\theta]])=0$. 

Now suppose $p=2$ satisfies condition \eqref{2}. We apply the Montes algorithm. Note that $2$ necessarily divides $a$, so modulo 2 we have
$$f_{a,b}(x)\equiv x^4+b\equiv (x+1)^4.$$
The $(x+1)$-adic development of $f_{a,b}(x)$ is 
$$f_{a,b}(x)=(x+1)^4-4(x+1)^3+6(x+1)^2+(a-4)(x+1)+b-a+1.$$
To show monogeneity, we need $\ind_{x+1}(f_{a,b})=0$. Thus we want $v_2(b-a+1)=1$. One checks this is equivalent to the criteria given in condition \eqref{2}. The residual polynomial is linear and hence separable. Thus, Theorem \ref{Thmofindex} tells us $v_2([\mathcal{O}_{\QQ(\theta)}:\ZZ[\theta]])=0$.

Finally, suppose $p=3$ satisfies condition \eqref{3}. We begin applying the Montes algorithm. Note that $3$ necessarily divides $b$, so that modulo 3 we have
$$f_{a,b}(x)\equiv x^4+ax\equiv x(x^3+a).$$
From Remark \ref{sep}, the separable factor $x$ contributes nothing to the index.

For the factor $(x^3+a)$, we have two cases:

{\bf Case 1:} Suppose $a\equiv 1$ modulo 3. Thus $f_{a,b}(x)\equiv x(x+1)^3$ modulo 3, we take the $(x+1)$-adic development
$$f_{a,b}(x)=(x+1)^4-4(x+1)^3+6(x+1)^2+(a-4)(x+1)+b-a+1.$$
In order to have $\ind_{x+1}(f_{a,b})=0$, we need $v_3(b-a+1)=1$. This is satisfied by the following pairs $(a,b)$ in $\ZZ/9\ZZ\times \ZZ/9\ZZ$:  $(1,3)$, $(1,6)$, $(4,0)$, $(4,6)$, $(7,0)$, $(7,3)$. The residual polynomial is linear and hence separable. Applying Theorem \ref{Thmofindex}, we conclude $v_3([\mathcal{O}_{\QQ(\theta)}:\ZZ[\theta]])=0$.

{\bf Case 2:} Suppose $a\equiv -1$ modulo 3. Thus $f_{a,b}\equiv x(x-1)^3$ modulo 3, we take the $(x-1)$-adic development
$$f_{a,b}(x)=(x-1)^4+4(x-1)^3+6(x-1)^2+(a+4)(x-1)+b+a+1.$$
In order to have $\ind_{x-1}(f_{a,b})=0$, we need $v_3(b+a+1)=1$. This is satisfied by the following pairs $(a,b)$ in $\ZZ/9\ZZ\times \ZZ/9\ZZ$:  $(2,0)$, $(2,3)$, $(5,0)$, $(5,6)$, $(8,3)$, $(8,6)$. The residual polynomials is linear and hence separable. Applying Theorem \ref{Thmofindex}, we conclude $v_3([\mathcal{O}_{\QQ(\theta)}:\ZZ[\theta]])=0$.

Since we have covered all primes dividing the discriminant of $f_{a,b}$, we see $[\mathcal{O}_{\QQ(\theta)}:\ZZ[\theta]]=1$. We conclude that $\QQ(\theta)$ is monogenic and $\theta$ generates the ring of integers. 
\end{proof}

Before proving Theorem \ref{gmon}, we remind ourselves of the statement:

\begin{theorem}\label{B}
Let $c$ and $d$ be integers such that $d$ is square-free and $256d-27c^4$ is not divisible by the square of an odd prime. If $4\mid \left(256d-27c^4\right)$, we require that $(c,d)$ is congruent to either $(0,1)$ or $(2,3)$ in $\ZZ/4\ZZ\times \ZZ/4\ZZ$. Suppose that $g_{c,d}(x)=x^4+cx^3+d$ is irreducible and let $\tau$ be a root. Then, $\QQ(\tau)$ is monogenic and $\tau$ is a generator of the ring of integers.
\end{theorem}

\begin{proof}
Recall $\Delta_g=d^2(256d-27c^4)$. Let $p$ be a prime dividing $\Delta_g$. We will show $v_p([\mathcal{O}_{\QQ(\tau)}:\ZZ[\tau]])=0$. First, suppose $p\mid \left(256d-27c^4\right)$, but $p\dnd d$ and $p\neq 2$. By assumption, $v_p(256d-27c^4)=1$. Hence
$$1=v_p(\Delta_g)=v_p(\disc(\QQ(\tau)))+2v_p([\mathcal{O}_{\QQ(\tau)}:\ZZ[\tau]]).$$
Thus $v_p([\mathcal{O}_{\QQ(\tau)}:\ZZ[\tau]])=0$.

Now suppose $p\mid d$. Applying the Montes algorithm, we consider $g_{c,d}(x)$ modulo $p$. We have two cases:

{\bf Case 1:}  Suppose $p\mid c$. The reduction of $g_{c,d}(x)$ is simply $x^4$, so we only consider the irreducible factor $x$. Taking the lift $\phi(x)=x$, the principal $x$-polygon of $g_{c,d}(x)$ has one side, originating at $(0,1)$ and terminating at $(4,0)$. Thus $\ind_x(g_{c,d})=0$. The residual polynomial attached to this side is $y-\frac{d}{p}$, which is clearly separable. Thus, by Theorem \ref{Thmofindex}, $v_p([\mathcal{O}_{\QQ(\tau)}:\ZZ[\tau]])=0$.

{\bf Case 2:}  Suppose $p\dnd c$. Modulo $p$ we have
$$g_{c,d}(x)\equiv x^4+cx^3\equiv x^3(x+c).$$ 
We treat the irreducible factor $x$ exactly as above. Again, the principal $x$-polygon is one-sided and the residual polynomial is separable. We conclude $\ind_x(g_{c,d})=0$.

Considering the factor $x+c$, we note it is separable. From Remark \ref{sep} we see $\ind_{x+c}(g_{c,d})=0$ and the residual polynomial is separable. We apply Theorem \ref{Thmofindex} to see $v_p([\mathcal{O}_{\QQ(\tau)}:\ZZ[\tau]])=0$.

For the final scenario, suppose $4\mid\left(256d-27c^4\right)$ and $2\dnd d$. Modulo 2 we have
$$g_{c,d}(x)\equiv x^4+d\equiv (x-1)^4.$$
Beginning the Montes algorithm, the $(x-1)$-adic development is 
$$(x-1)^4+(c+4)(x-1)^3+(3c+6)(x-1)^2+(3c+4)(x-1)+c+d+1.$$
To ensure $\ind_{x-1}(g_{c,d})=0$, we need $v_2(c+d+1)=1$. One checks this is equivalent to the conditions given in the theorem statement. Finally, if $v_2(c+d+1)=1$ the residual polynomial is linear and hence separable. Thus, by Theorem \ref{Thmofindex}, $v_2([\mathcal{O}_{\QQ(\tau)}:\ZZ[\tau]])=0$.

Since we have covered all primes dividing the discriminant of $g_{c,d}$, we conclude $[\mathcal{O}_{\QQ(\tau)}:\ZZ[\tau]]=1$. Thus $\QQ(\tau)$ is monogenic and $\tau$ generates the ring of integers. 

\end{proof}

\section{Density}\label{Density}
In this section, we will show the families of monogenic $S_4$ fields defined by the polynomials $f_{b,b}(x)=x^4+bx+b$ and $g_{1,d}(x)=x^4+x^3+d$ with the conditions imposed in Theorem \ref{Gal} are infinite. In fact, we will give a lower bound on the density of each family. The families $f_{b,b}$ and $g_{1,d}$ are parametrized by $b$ and $d$ respectively. So, by \emph{density}, we mean the natural density of $b\in\ZZ$ or $d\in\ZZ$ yielding monogenic fields.

To begin with, it is well-known that the natural density of square-free integers is 
$$\dfrac{1}{\zeta(2)}=\dfrac{6}{\pi^2}\approx 60.79\%.$$
See \cite{Jia} for example. Now let $S(x;m,k)$ denote the number of square-free integers that do not exceed $x$ and are congruent to $m$ modulo $k$. We will also need a result of Prachar from \cite{Prachar}:

\begin{theorem}\label{Prachar}
$$S(x;m,k)\sim \dfrac{6x}{\pi^2k}\prod\limits_{p\mid k}\left(1-\dfrac{1}{p^2}\right)^{-1} \ \ \ \ \ (x\to \infty)$$
for $\gcd(m,k)=1$ and $k\leq x^{\frac{2}{3}-\epsilon}$. 
\end{theorem}

Before the proof, we recall the second claim of Theorem \ref{f1}:
\begin{theorem}
The density of monogenic $S_4$ fields within the number fields defined by $f_{b,b}(x)=x^4+bx+b$ is at least $\dfrac{51-4\pi^2}{4\pi^2}\approx 29.18\%$. 
\end{theorem}

\begin{proof}
From Theorem \ref{A}, to show that there are infinitely many monogenic fields defined by a root of $f_{b,b}$, it suffices to show there are infinitely many square-free $b$ such that $256-27b$ is square-free. The density of square-free $b$ is $\dfrac{6}{\pi^2}$. By Theorem \ref{Prachar}, the density of square-free numbers congruent to $256$ modulo $27$ among numbers congruent to $256$ modulo $27$ is 
$$\dfrac{6}{\pi^2}\left(1-\dfrac{1}{9}\right)^{-1}=\dfrac{27}{4\pi^2}.$$
Thus, at worst, the density of monogenic fields in this family is 
$$\dfrac{6}{\pi^2}-\left(1-\dfrac{27}{4\pi^2}\right)=\dfrac{51-4\pi^2}{4\pi^2}\approx 29.18\%.$$
\end{proof}

As above, before beginning the proof, we recall the second claim of Theorem \ref{g1}:
\begin{theorem}
The density of monogenic $S_4$ fields within the number fields defined by $g_{1,d}(x)=x^4+x^3+d$ is at least $\dfrac{14-\pi^2}{\pi^2}\approx 41.85\%$.
\end{theorem}

\begin{proof} 
From Theorem \ref{B}, to show that there are infinitely many monogenic fields defined by a root of $g_{1,d}$, it suffices to show there are infinitely many square-free $d$ such that $256d-27$ is square-free. The density of square-free $d$ is $\dfrac{6}{\pi^2}$. By Theorem \ref{Prachar}, the density of square-free numbers congruent to $27$ modulo $256$ among numbers congruent to $27$ modulo $256$ is 
$$\dfrac{6}{\pi^2}\left(1-\dfrac{1}{4}\right)^{-1}=\dfrac{24}{3\pi^2}.$$
Thus, at worst, the density of monogenic fields in this family is 
$$\dfrac{6}{\pi^2}-\left(1-\dfrac{24}{3\pi^2}\right)=\dfrac{14-\pi^2}{\pi^2}\approx 41.85\%.$$
\end{proof}

\begin{remark}\label{data}
As above, let $\theta$ be a root of $f_{b,b}(x)=x^4+bx+b$ and $\tau$ a root of $g_{1,d}(x)=x^4+x^3+d$. Computationally, it appears that 55.3\% of fields of the form $\QQ(\theta)$ have $\theta$ as a generator of $\mathcal{O}_{\QQ(\theta)}$. Likewise, it appears that 55.3\% of fields of the form $\QQ(\tau)$ have $\tau$ as a generator of $\mathcal{O}_{\QQ(\tau)}$. If $\QQ(\tau)$ is monogenic, it seems that $\tau$ is almost always a generator of the ring of integers, since $\QQ(\tau)$ appears to be monogenic about 55.3\% of the time. However, $\QQ(\theta)$ seems to be monogenic about 58.7\% of the time. Thus there are some cases where $\QQ(\theta)$ is monogenic, but $\theta$ does not generate the ring of integers. We obtained these heuristics using SageMath \cite{Sage} and testing $b$ and $d$ between -2,500,000 and 2,500,000. 
\end{remark}

\bibliography{Bib}

\begin{thebibliography}{33}
\providecommand{\natexlab}[1]{#1}
\providecommand{\url}[1]{\texttt{#1}}
\expandafter\ifx\csname urlstyle\endcsname\relax
  \providecommand{\doi}[1]{doi: #1}\else
  \providecommand{\doi}{doi: \begingroup \urlstyle{rm}\Url}\fi

\bibitem[B\'erczes et~al.(2013)B\'erczes, Evertse, and Gy\H{o}ry]{BEG}
A.~B\'erczes, J.-H. Evertse, and K.~Gy\H{o}ry.
\newblock Multiply monogenic orders.
\newblock \emph{Ann. Sc. Norm. Super. Pisa Cl. Sci. (5)}, 12\penalty0
  (2):\penalty0 467--497, 2013.
\newblock ISSN 0391-173X.

\bibitem[{Bhargava} et~al.(2016){Bhargava}, {Shankar}, and {Wang}]{BSW}
M.~{Bhargava}, A.~{Shankar}, and X.~{Wang}.
\newblock {Squarefree values of polynomial discriminants I}.
\newblock \emph{ArXiv e-prints}, Nov. 2016.
\newblock URL \url{https://arxiv.org/abs/1611.09806}.

\bibitem[Brookfield(2007)]{LostArt}
G.~Brookfield.
\newblock Factoring quartic polynomials: a lost art.
\newblock \emph{Mathematics Magazine}, 80\penalty0 (1):\penalty0 67--70, 2007.
\newblock URL \url{http://www.jstor.org/stable/27642994}.

\bibitem[Conrad()]{KCBlurb}
K.~Conrad.
\newblock Galois groups of cubics and quartics (not in characteristic 2).
\newblock URL
  \url{http://www.math.uconn.edu/~kconrad/blurbs/galoistheory/cubicquartic.pdf}.

\bibitem[Dedekind(1878)]{Dedekind}
R.~Dedekind.
\newblock \"{U}ber den {Z}usammenhang zwischen der {T}heorie der {I}deale und
  der {T}heorie der h\"{o}heren {K}ongruenzen.
\newblock \emph{G\"{o}tt. Abhandlungen}, pages 1--23, 1878.

\bibitem[El~Fadil et~al.(2012)El~Fadil, Montes, and Nart]{EFMN}
L.~El~Fadil, J.~Montes, and E.~Nart.
\newblock Newton polygons and {$p$}-integral bases of quartic number fields.
\newblock \emph{J. Algebra Appl.}, 11\penalty0 (4):\penalty0 1250073, 33, 2012.
\newblock ISSN 0219-4988.
\newblock \doi{10.1142/S0219498812500739}.
\newblock URL \url{http://dx.doi.org/10.1142/S0219498812500739}.

\bibitem[Funakura(1984)]{Funa}
T.~Funakura.
\newblock On integral bases of pure quartic fields.
\newblock \emph{Math. J. Okayama Univ.}, 26:\penalty0 27--41, 1984.
\newblock ISSN 0030-1566.

\bibitem[Ga\'al(1993)]{Gaalquartic}
I.~Ga\'al.
\newblock Power integral bases in orders of families of quartic fields.
\newblock \emph{Publ. Math. Debrecen}, 42\penalty0 (3-4):\penalty0 253--263,
  1993.
\newblock ISSN 0033-3883.

\bibitem[Ga\'al(2002)]{GaalDiophantine}
I.~Ga\'al.
\newblock \emph{Diophantine equations and power integral bases}.
\newblock Birkh\"auser Boston, Inc., Boston, MA, 2002.
\newblock ISBN 0-8176-4271-4.
\newblock \doi{10.1007/978-1-4612-0085-7}.
\newblock URL \url{https://doi.org/10.1007/978-1-4612-0085-7}.
\newblock New computational methods.

\bibitem[Ga\'{a}l and Jadrijevi\'{c}(2017)]{GaalJadrijevic}
I.~Ga\'{a}l and B.~Jadrijevi\'{c}.
\newblock Determining elements of minimal index in an infinite family of
  totally real bicyclic biquadratic number fields.
\newblock \emph{JP J. Algebra Number Theory Appl.}, 39:\penalty0 307--326, 05
  2017.

\bibitem[Ga\'{a}l and Remete(2014)]{GaalRemete}
I.~Ga\'{a}l and L.~Remete.
\newblock Binomial thue equations and power integral bases in pure quartic
  fields.
\newblock \emph{JP J. Algebra Number Theory Appl.}, 32:\penalty0 49--61, 02
  2014.

\bibitem[Ga\'al and Szab\'o(2012)]{GaalSzabo}
I.~Ga\'al and T.~Szab\'o.
\newblock Power integral bases in parametric families of biquadratic fields.
\newblock \emph{JP J. Algebra Number Theory Appl.}, 24\penalty0 (1):\penalty0
  105--114, 2012.
\newblock ISSN 0972-5555.

\bibitem[Ga\'al et~al.(1991{\natexlab{a}})Ga\'al, Peth\H{o}, and Pohst]{GPPI}
I.~Ga\'al, A.~Peth\H{o}, and M.~Pohst.
\newblock On the resolution of index form equations in biquadratic number
  fields. {I}, {II}.
\newblock \emph{J. Number Theory}, 38\penalty0 (1):\penalty0 18--34, 35--51,
  1991{\natexlab{a}}.
\newblock ISSN 0022-314X.
\newblock \doi{10.1016/0022-314X(91)90090-X}.
\newblock URL \url{https://doi.org/10.1016/0022-314X(91)90090-X}.

\bibitem[Ga\'al et~al.(1991{\natexlab{b}})Ga\'al, Peth\H{o}, and
  Pohst]{GPPindexform}
I.~Ga\'al, A.~Peth\H{o}, and M.~Pohst.
\newblock On the resolution of index form equations.
\newblock In S.~M. Watt, editor, \emph{Proceedings of the 1991 International
  Symposium on Symbolic and Algebraic Computation}, pages 185--186, New York,
  NY, USA, 1991{\natexlab{b}}. ACM.
\newblock ISBN 0-89791-437-6.

\bibitem[Ga\'al et~al.(1991{\natexlab{c}})Ga\'al, Peth\H{o}, and
  Pohst]{GaalPethoPohst}
I.~Ga\'al, A.~Peth\H{o}, and M.~Pohst.
\newblock On the indices of biquadratic number fields having {G}alois group
  {$V_4$}.
\newblock \emph{Arch. Math. (Basel)}, 57\penalty0 (4):\penalty0 357--361,
  1991{\natexlab{c}}.
\newblock ISSN 0003-889X.
\newblock \doi{10.1007/BF01198960}.
\newblock URL \url{https://doi.org/10.1007/BF01198960}.

\bibitem[Gassert(2017)]{Alden}
T.~A. Gassert.
\newblock A note on the monogeneity of power maps.
\newblock \emph{Albanian J. Math.}, 11\penalty0 (1):\penalty0 3--12, 2017.
\newblock ISSN 1930-1235.

\bibitem[{Gassert} et~al.(2017){Gassert}, {Smith}, and {Stange}]{GSS}
T.~A. {Gassert}, H.~{Smith}, and K.~E. {Stange}.
\newblock {A family of monogenic $S_4$ quartic fields arising from elliptic
  curves}.
\newblock \emph{ArXiv e-prints}, Aug. 2017.
\newblock URL \url{https://arxiv.org/abs/1708.03953}.

\bibitem[Gras(1981)]{GrasZbases}
M.-N. Gras.
\newblock {${\bf Z}$}-bases d'entiers {$1,$} {$\theta ,$} {$\theta ^{2},$}
  {$\theta ^{3}$}\ dans les extensions cycliques de degr\'e {$4$}\ de {${\bf
  Q}$}.
\newblock In \emph{Number theory, 1979--1980 and 1980--1981}, Publ. Math. Fac.
  Sci. Besan\c{c}on, pages Exp. No. 6, 14. Univ. Franche-Comt\'e, Besan\c{c}on,
  1981.

\bibitem[Gras(1986)]{Gras}
M.-N. Gras.
\newblock Condition n\'ecessaire de monog\'en\'eit\'e de l'anneau des entiers
  d'une extension ab\'elienne de {${\bf Q}$}.
\newblock In \emph{S\'eminaire de th\'eorie des nombres, {P}aris 1984--85},
  volume~63 of \emph{Progr. Math.}, pages 97--107. Birkh\"auser Boston, Boston,
  MA, 1986.

\bibitem[Gras and Tano\'e(1995)]{tangras}
M.-N. Gras and F.~Tano\'e.
\newblock Corps biquadratiques monog\`enes.
\newblock \emph{Manuscripta Math.}, 86\penalty0 (1):\penalty0 63--79, 1995.
\newblock ISSN 0025-2611.
\newblock \doi{10.1007/BF02567978}.
\newblock URL \url{http://dx.doi.org/10.1007/BF02567978}.

\bibitem[Gu\`ardia et~al.(2015)Gu\`ardia, Montes, and Nart]{gmn15}
J.~Gu\`ardia, J.~Montes, and E.~Nart.
\newblock Higher {N}ewton polygons and integral bases.
\newblock \emph{J. Number Theory}, 147:\penalty0 549--589, 2015.
\newblock ISSN 0022-314X.
\newblock \doi{10.1016/j.jnt.2014.07.027}.
\newblock URL
  \url{http://0-dx.doi.org.libraries.colorado.edu/10.1016/j.jnt.2014.07.027}.

\bibitem[Halter-Koch and Tichy(2000)]{Hasse}
F.~Halter-Koch and R.~F. Tichy, editors.
\newblock \emph{Algebraic number theory and {D}iophantine analysis}, 2000.
  Walter de Gruyter \& Co., Berlin.
\newblock ISBN 3-11-016304-7.
\newblock \doi{10.1515/9783110801958}.
\newblock URL \url{https://doi.org/10.1515/9783110801958}.

\bibitem[Huard et~al.(1995)Huard, Spearman, and Williams]{HSW}
J.~G. Huard, B.~K. Spearman, and K.~S. Williams.
\newblock Integral bases for quartic fields with quadratic subfields.
\newblock \emph{J. Number Theory}, 51\penalty0 (1):\penalty0 87--102, 1995.
\newblock ISSN 0022-314X.
\newblock \doi{10.1006/jnth.1995.1036}.
\newblock URL \url{http://dx.doi.org/10.1006/jnth.1995.1036}.

\bibitem[Jadrijevi\'c(2009)]{Jadrijevic}
B.~Jadrijevi\'c.
\newblock Solving index form equations in two parametric families of
  biquadratic fields.
\newblock \emph{Math. Commun.}, 14\penalty0 (2):\penalty0 341--363, 2009.
\newblock ISSN 1331-0623.

\bibitem[Jadrijevi\'c(2015)]{Jadrijevic1}
B.~Jadrijevi\'c.
\newblock On elements with index of the form {$2^a3^b$} in a parametric family
  of biquadratic fields.
\newblock \emph{Glas. Mat. Ser. III}, 50(70)\penalty0 (1):\penalty0 43--63,
  2015.
\newblock ISSN 0017-095X.
\newblock \doi{10.3336/gm.50.1.05}.
\newblock URL \url{https://doi.org/10.3336/gm.50.1.05}.

\bibitem[Jia(1993)]{Jia}
C.~H. Jia.
\newblock The distribution of square-free numbers.
\newblock \emph{Sci. China Ser. A}, 36\penalty0 (2):\penalty0 154--169, 1993.
\newblock ISSN 1001-6511.

\bibitem[Jones and Phillips(2018)]{JonesPhillips}
L.~Jones and T.~Phillips.
\newblock Infinite families of monogenic trinomials and their {G}alois groups.
\newblock \emph{Internat. J. Math.}, 29\penalty0 (5):\penalty0 1850039, 11,
  2018.
\newblock ISSN 0129-167X.
\newblock \doi{10.1142/S0129167X18500398}.
\newblock URL \url{https://doi.org/10.1142/S0129167X18500398}.

\bibitem[Kable(1999)]{Kable}
A.~C. Kable.
\newblock Power bases in dihedral quartic fields.
\newblock \emph{J. Number Theory}, 76\penalty0 (1):\penalty0 120--129, 1999.
\newblock ISSN 0022-314X.
\newblock \doi{10.1006/jnth.1998.2350}.
\newblock URL \url{http://dx.doi.org/10.1006/jnth.1998.2350}.

\bibitem[Kappe and Warren(1989)]{GalQuar}
L.-C. Kappe and B.~Warren.
\newblock An elementary test for the {G}alois group of a quartic polynomial.
\newblock \emph{Amer. Math. Monthly}, 96\penalty0 (2):\penalty0 133--137, 1989.
\newblock ISSN 0002-9890.
\newblock URL \url{https://doi.org/10.2307/2323198}.

\bibitem[Olajos(2005)]{Olajos}
P.~Olajos.
\newblock Power integral bases in the family of simplest quartic fields.
\newblock \emph{Experiment. Math.}, 14\penalty0 (2):\penalty0 129--132, 2005.
\newblock ISSN 1058-6458.
\newblock URL \url{http://projecteuclid.org/euclid.em/1128100125}.

\bibitem[Prachar(1958)]{Prachar}
K.~Prachar.
\newblock \"{U}ber die kleinste quadratfreie {Z}ahl einer arithmetischen
  {R}eihe.
\newblock \emph{Monatsh. Math.}, 62:\penalty0 173--176, 1958.
\newblock URL \url{https://doi.org/10.1007/BF01301288}.

\bibitem[Spearman(2006)]{spear}
B.~K. Spearman.
\newblock Monogenic {$A_4$} quartic fields.
\newblock \emph{Int. Math. Forum}, 1\penalty0 (37-40):\penalty0 1969--1974,
  2006.
\newblock ISSN 1312-7594.
\newblock \doi{10.12988/imf.2006.06174}.
\newblock URL \url{http://dx.doi.org/10.12988/imf.2006.06174}.

\bibitem[{The Sage Developers}(2017)]{Sage}
{The Sage Developers}.
\newblock \emph{{S}ageMath, the {S}age {M}athematics {S}oftware {S}ystem
  ({V}ersion 8.1)}, 2017.
\newblock \url{http://www.sagemath.org}.

\end{thebibliography}
\bibliographystyle{abbrvnat}

\end{document}